\documentclass[12pt,a4paper]{amsart}

\usepackage[utf8]{inputenc}
\usepackage[margin=1in]{geometry}
\usepackage{thmtools} 
\usepackage{hyperref}
\usepackage{etoolbox} 
\usepackage[capitalize,noabbrev]{cleveref}
\usepackage{amsmath}
\usepackage{amsfonts}
\usepackage{amsthm}
\usepackage{amssymb}
\usepackage{xcolor}
\usepackage{graphicx}
\usepackage{tikz} 
\usetikzlibrary{positioning}
\usepackage{enumerate} 
\usepackage{subcaption} 
\usepackage{mathrsfs} 
\usepackage{enumitem}
\usepackage{enumerate}
\usepackage{comment}
\usepackage{multicol}
\usepackage{url}
\usepackage{stackengine}
\usepackage{nicematrix}
\usepackage{fancyhdr}
\usepackage{placeins}

\theoremstyle{plain}
\newtheorem{theorem}{Theorem}[section]
\def\bthm{\begin{theorem}}
\def\ethm{\end{theorem}}

\newtheorem{theoremx}{Theorem}

\newtheorem{prop}[theorem]{Proposition}
\def\bprop{\begin{prop}}
\def\eprop{\end{prop}}

\newtheorem{ques}[theorem]{Question}
\def\bques{\begin{ques}}
\def\eques{\end{ques}}

\newtheorem{cjec}[theorem]{Conjecture}
\def\bcjec{\begin{cjec}}
\def\ecjec{\end{cjec}}

\newtheorem{cor}[theorem]{Corollary}
\def\bcor{\begin{cor}}
\def\ecor{\end{cor}}

\newtheorem{fact}[theorem]{Fact}
\def\bfact{\begin{fact}}
\def\efact{\end{fact}}

\newtheorem{lem}[theorem]{Lemma}
\newtheorem{lemma}[theorem]{Lemma}
\def\blem{\begin{lem}}
\def\elem{\end{lem}}

\theoremstyle{definition}
\newtheorem{defn}[theorem]{Definition}
\def\bdefn{\begin{defn}}
\def\edefn{\end{defn}}

\newtheorem{nota}[theorem]{Notation}
\def\bnota{\begin{nota}}
\def\enota{\end{nota}}

\newtheorem*{conc}{Conclusion}
\def\bconc{\begin{conc}}
\def\econc{\end{conc}}

\newtheorem{alg}[theorem]{Algorithm}
\def\balg{\begin{alg}}
\def\ealg{\end{alg}}

\def\bproof{\begin{proof}}
\def\eproof{\end{proof}}

\theoremstyle{remark}

\newtheorem{rem}[theorem]{Remark}
\def\brem{\begin{rem}}
\def\erem{\end{rem}}

\newtheorem{ex}[theorem]{Example}
\def\bex{\begin{ex}}
\def\eex{\end{ex}}

\newtheorem{exs}[theorem]{Examples}
\def\bexs{\begin{exs}}
\def\eexs{\end{exs}}

\newtheorem{obs}{Observation}
\def\bobs{\begin{obs}}
\def\eobs{\end{obs}}	

\makeatletter
\let\c@figure\c@theorem

\makeatother

\let\ra=\rightarrow

\let\iso=\cong

\let\tensor=\otimes

\let\epsilon=\varepsilon

\def\Z{\mathbb{Z}}

\def\a{\mathbf{a}}
\def\b{\mathbf{b}}
\def\c{\mathbf{c}}
\def\e{\mathbf{e}}
\def\o{\mathbf{1}}
\def\u{\mathbf{u}}
\def\eps{\boldsymbol{\epsilon}}
\def\brho{\boldsymbol{\rho}}
\def\k{\mathbb{K}}

\DeclareMathOperator{\cc}{\mathbf{c}}
\DeclareMathOperator{\Ac}{A_\mathbf{c}}
\DeclareMathOperator{\Span}{Span}

\newcommand{\norm}[1]{|| #1 ||}

\newcommand{\card}[1]{\left \lvert #1 \right\rvert}

\newcommand{\ul}[1]{\underline{#1}}
\newcommand{\mscr}[1]{\mathscr{#1}}

\renewcommand{\subset}{\subseteq}

\newcommand{\NL}[2]{{\text{NL}}{\left(#1,#2\right)}}

\begin{document}

\title[WLP of Monomial Complete Intersections]{The Non-Lefschetz Locus of Monomial Complete Intersections in Positive Characteristic}
\author[A. LaClair]{Adam LaClair}
\address{University of Nebraska--Lincoln, Department of Mathematics, Lincoln, NE, USA}
\email{alaclair2@unl.edu}


\begin{abstract}
In the 1980s, R. Stanley and J.\ Watanabe proved that over a field of characteristic zero every monomial complete intersection has the strong Lefschetz property. Determining which monomial complete intersections in positive characteristic have the weak Lefschetz property has been the subject of extensive research. In this paper, we give a complete description of the locus of exponent vectors corresponding to monomial complete intersections over a field of positive characteristic that fail the weak Lefschetz property. We describe this locus explicitly in terms of translates of positive-orthant lattice balls in the taxicab metric.
\end{abstract}


\maketitle 
\pagestyle{fancy}
\fancyhf{}

\fancyhead[LE]{\normalfont\scriptsize\thepage}
\fancyhead[CE]{\normalfont\scriptsize A. LACLAIR}

\fancyhead[CO]{\normalfont\scriptsize WLP OF MONOMIAL COMPLETE INTERSECTIONS}
\fancyhead[RO]{\normalfont\scriptsize\thepage}

\renewcommand{\headrulewidth}{0pt}

\section{Introduction}

Let $\k$ denote a field of positive characteristic $p > 0$. For $n \geq 3$ and $\cc \in \Z_{\geq 2}^n$, consider the Artinian monomial complete intersection
\begin{align*}
    A_{\cc} := \frac{\k[x_1,\ldots,x_n]}{(x_1^{c_1},\ldots,x_n^{c_n})}.
\end{align*}
$A_{\cc}$ is said to have the \emph{weak Lefschetz property} (WLP) if there exists a linear form $\ell \in A_{\cc}$ so that 
\begin{align*}
    \times \ell : [A_{\cc}]_i \ra [A_{\cc}]_{i+1}
\end{align*}
has maximal rank (i.e., injective or surjective) for every $i \geq 0$. $A_{\cc}$ is said to have the \emph{strong Lefschetz property} (SLP) if multiplication by $\ell^d$ has maximal rank for all positive integers $d$ and for all $i \geq 0$. Stanley \cite{Stanley1980HardLefschetz} and, independently, Watanabe \cite{watanabe1987dilworth} proved that $A_{\cc}$ has the strong Lefschetz property in characteristic $0$. These results have motivated extensive research into understanding which $\k$-algebras possess the weak or strong Lefschetz properties and their connections to other aspects of commutative algebra, as well as the study of these properties in positive characteristic \cite{harima2003weak,MiglioreMiroRoig2003,MaenoWatanabe2009,MiglioreNagel2013}. 

In positive characteristic, determining either Lefschetz property is considerably more subtle. In 2019, Lundqvist and
Nicklasson~\cite{LundqvistNicklasson2019}, building on work of
Cook II~\cite{Cook2012}, gave a complete numerical characterization
of when $\Ac$ has the SLP for $n \geq 3$. In three variables, Li and Zanello~\cite{LiZanello2010} characterized the WLP for
monomial complete intersections. Subsequent work on the WLP has established classifications for special families of exponent sequences \cite{BrennerKaid2011, KustinVraciu2014, LundqvistNicklasson2019, Vraciu2015}. In 2024, Kyomuhangi, Marangone, Raicu, and Reed~\cite{kyomuhangi2024cohomology} translated the WLP into a question about weights in cohomology representations associated with the incidence correspondence and characterized failure of the WLP in characteristic $2$. In~\cite{KMRR2026Cohomology} and~\cite{KMRR2026Lefschetz}, the authors extended their techniques to give a complete numerical characterization of when $\Ac$ has the WLP in positive characteristic. 

In this paper, we give an explicit geometric description of the set of exponent vectors for which $A_{\mathbf c}$ fails the WLP. Our formulation expresses this set as a union of prescribed translates of positive-orthant taxicab lattice balls, with explicit radii and translation vectors. We obtained this characterization independently of \cite{KMRR2026Cohomology, KMRR2026Lefschetz}, and it is equivalent to their numerical characterization. Our approach uses combinatorial techniques to construct kernel elements witnessing failure of the WLP. For the reverse inclusion, we use the structural result {\cite[Proposition 5.4]{kyomuhangi2024cohomology}, now appearing as \cite[Proposition 3.3]{KMRR2026Cohomology}}, to prove that every failing exponent vector belongs to one of the prescribed translates.

Let $\NL{n}{p}$ denote the set of exponent vectors $\cc \in \Z_{\geq 2}^n$ for which $A_{\cc}$ fails the weak Lefschetz property. Write $\norm{\cdot}_1$ for the taxicab norm, and define 
\begin{equation}
    \label{eqn:gamma_n}
    \Gamma_n := \left\{ \b \in \Z_{\geq0}^n : \norm{\b}_1 \text{ is even and } 2 \max_i b_i \leq \norm{\b}_1 \right\}.
\end{equation}
For $\mathbf{a} \in \Z^{n}$ and $r \in \Z_{\geq 0}$, let
\begin{align*}
    B_r^+(\mathbf{a}) := \{\mathbf{u} \in (\Z_{\geq 1})^{n} : \norm{\mathbf{u} - \mathbf{a}}_1 \leq r \}
\end{align*}
be the \emph{positive-orthant lattice ball} of radius $r$ centered at $\a$. For $q = p^e$ where $e\geq 1$, put 
\begin{equation}
    r_q :=(n-2)(q-1)-1
\end{equation}
and write $\o = (1,\ldots,1).$

Our main result is then the following geometric characterization of the non-weak Lefschetz locus of exponent vectors of monomial complete intersections in positive characteristic.
\begin{theoremx}
\label{thm:intro}
For $n \geq 3$ and a prime $p > 0$, 
\begin{align*}
    \NL{n}{p} = \Z_{\geq 2}^n \cap \bigcup_{e\geq 1}\bigcup_{\b\in \Gamma_n} \left( B_{r_{p^e}}^+(p^e\o) + p^e\b \right).
\end{align*}
\end{theoremx}
Theorem \ref{thm:intro} shows that the non-Lefschetz locus can be expressed as a union of translates of positive-orthant lattice balls. Figure~\ref{fig:NL33} illustrates this phenomenon for $n=3$ and $p=3$.

The two inclusions have complementary proofs. For the forward inclusion, we apply Frobenius to a suitable relation in an auxiliary monomial complete intersection. The taxicab bound is the degree estimate ensuring that the resulting construction witnesses failure of the WLP. For the reverse inclusion, we start with a graded $K[T]$-summand witnessing failure of the WLP and restrict to the action of $T^q$ for a suitable prime power $q$. This produces a summand of length coprime to $p$ in an auxiliary monomial complete intersection. Lemma \ref{lem:prime_to_p_centering} controls its grading, and an elementary adjustment of the resulting exponent vector places the original vector in one of the prescribed translates.

In the next section, we introduce the relevant background and prove Theorem \ref{thm:intro}.

\section{Results}

We begin with a few preliminary results.

\begin{lemma}
\label{rem:translated_equiv}
We have that $\b \in \Gamma_n$ if and only if there exists a non-negative integer $d$ and pairs of integers $1 \leq u_j < v_j \leq n$ for $1 \leq j \leq d$ so that
\begin{align*}
    \b = \sum_{j=1}^d (\e_{u_j} + \e_{v_j})
\end{align*}
where $\e_{k}$ denotes the $k$-th standard basis vector of $\Z_{\ge0}^n$.
\end{lemma}

\begin{proof}
The proof of the forward direction is by induction on $\norm{\b}_1$. For the inductive step, we subtract $1$ from the largest positive coordinates and show that this vector belongs to $\Gamma_n$.
\end{proof}

For convenience, we define the set
\begin{align*}
    \mathscr{B}_q := \Z_{\geq 2}^n \cap \bigcup_{\b\in \Gamma_n} \left( B_{r_q}^+(q\o) + q\b \right).
\end{align*}
By $B_{r_q}^+(q\o) + q\b$, we mean the pointwise translation of each vector of $B_{r_q}^+(q\o)$ by $q\b$. 

\begin{lemma}
\label{lem:equiv_char}
We have that 
\end{lemma}
\begin{equation}
\label{eq:descr_Bq}
 \c \in \mscr{B}_q \iff 
\begin{cases}
 \c\in \Z_{\geq 2}^n, \\
 \c = q(\o + \b) + \eps \text{ for some } \b \in \Gamma_n \text{ and } \eps \in \Z^n, \\
 \norm{\eps}_1 \leq r_q, \\
 \epsilon_i \geq 1-q \text{ for } 1 \leq i \leq n.
\end{cases}
\end{equation}

\begin{rem}
Let $\c \in \Z^n_{\geq 2}$. If $\c \in \mathscr{B}_q$, then
\begin{align*}
    \sigma := \sum_{i=1}^n(c_i -1) \geq 2q-1.
\end{align*}
Equivalently, if $\c \in \mscr{B}_q$, then $q \leq \frac{\sigma+1}{2}$. For a fixed vector $\c$, determining whether $\c \in \NL{n}{p}$ becomes a finite check. This same bound appears in {\cite[equation (1.6)]{KMRR2026Lefschetz}}.
\end{rem}

\begin{ex}
Let $n = 3$, $p = q = 3$, and $\b = (4,4,4) \in \Gamma_n$. Then, $r_3 = 1$, and we obtain that the following translate of the taxicab lattice ball centered at 
\begin{align*}
    q\o + q\b = (15,15,15)
\end{align*}
is contained in the non-Lefschetz locus of exponent vectors,
\begin{align*}
    \{ \c \in \Z_{\geq2}^3 : \card{c_1-15}+\card{c_2-15}+\card{c_3-15}  \leq 1\} \subset \NL{3}{3}.
\end{align*}
In particular, $(14,15,15)$ lies on the boundary of the above ball, and we have that
\begin{align*}
    \mathbb{F}_3[x,y,z]/(x^{14},y^{15},z^{15})
\end{align*}
fails the WLP. Figure~\ref{fig:NL33} depicts the subset of $\NL{3}{3}$ with $2 \leq c_i \leq 16$.
\end{ex}

\begin{figure}[!htbp]
\centering
\includegraphics[width=0.8\textwidth]{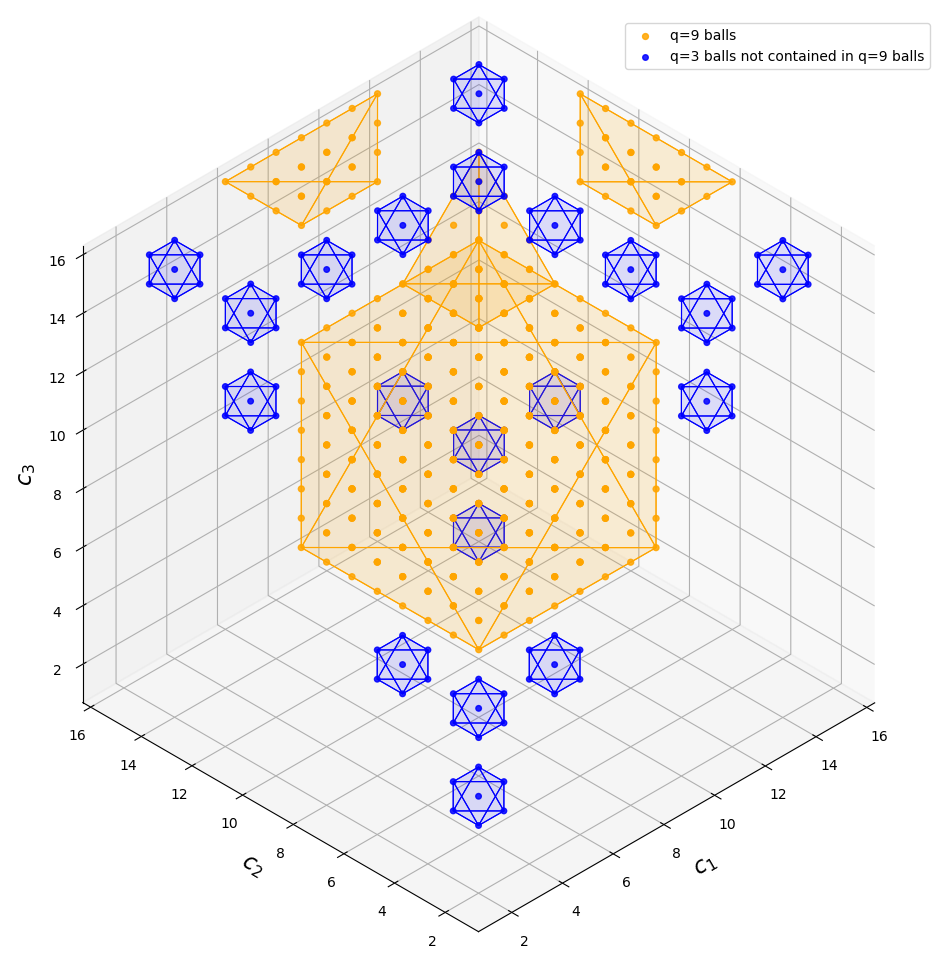} 
\caption{Convex hulls of the $q=3$ (blue) and $q=9$ (orange) translated taxicab balls intersecting the region $2\leq c_i\leq 16$. The $q=3$ balls contained entirely in a $q=9$ ball are omitted.
}
\label{fig:NL33}
\end{figure}

\subsection{Forward Inclusion}

We begin by recording the following result on the Hilbert function of a monomial complete intersection whose vector of exponents minus one belongs to $\Gamma_n$.

For a non-negative integer $a$, we define $[a] := \{0 < 1 < \cdots < a\}$ to be the chain of length $a$. We recall the following combinatorial lemma. For $m \geq 0$, let 
\begin{align*}
    P_m(t) := \sum_{i=0}^m t^i \in \Z[t].
\end{align*}

\FloatBarrier

\blem
\label{lem:HF_ineq}
Let $b_1,\ldots,b_n$ be non-negative integers satisfying:
\begin{enumerate}
    \item $\sum_{i=1}^n b_i = 2d$ for some non-negative integer $d$, and
    \item $b_i \leq d$ for all $1 \leq i \leq n$.
\end{enumerate}
If 
\begin{align*}
    \prod_{i=1}^n P_{b_i}(t) = \sum_{i=0}^{2d} h_i t^i
\end{align*}
and $h_{-1} := 0$, then $h_d > h_{d-1}$.
\elem 

\bproof
When $d = 0$, $h_0 = 1 > 0 = h_{-1}$. Hence we reduce to the case that $d \geq 1$. Without loss of generality, we may suppose that $b_1 \geq \cdots \geq b_n \geq 0$. Let $P$ be the product of chains poset $[b_1] \times \cdots \times [b_n]$. The number of elements of $P$ of rank $i$ is given by $h_i$. We have
\begin{align*}
    b_1 \leq d < d+1 \leq \left( \sum_{i=2}^n b_i\right)+1.
\end{align*}
Hence, Dhand's strict unimodality criterion for products of chains {\cite[Lemma 2.1(3)]{Dhand2014}} implies that $h_d > h_{d-1}$.
\eproof

We are now ready to prove that $\mscr{B}_q \subset \NL{n}{p}$.

\begin{proof}[Proof of the forward inclusion of Theorem~\ref{thm:intro}]
Let $\c \in \mscr{B}_q$. Then, $\c = q(\o+\b)+\eps$ as in \eqref{eq:descr_Bq}. Let $S = \k[x_1,\ldots,x_n]$ and $J := (x_1^{b_1+1},\ldots,x_n^{b_n+1})$. The Hilbert series of $S/J$ is given by $\prod_{i=1}^{n} P_{b_i}(t)$, and the socle degree of $S/J$ is $\sum_{i=1}^{n} b_i = 2d$ for some non-negative integer $d$. By Lemma \ref{lem:HF_ineq} together with symmetry of the Hilbert function, we have that 
\begin{align*}
    \dim_{\k} [S/J]_d > \dim_{\k} [S/J]_{d+1}.
\end{align*}
Thus, there exists $f \in [S]_d$ satisfying $f \notin J$ and $Lf \in J$ where $L = x_1+\cdots+x_n$.

For $1 \leq i \leq n$, put $\epsilon^+_i := \max\{0,\epsilon_i\}$ and $\epsilon_i^- := \min\{0,\epsilon_i\}$. Then, $\epsilon_i = \epsilon_i^+ + \epsilon_i^-$ and $\norm{\eps}_1 = \sum_{i=1}^n (\epsilon_i^+ - \epsilon_i^-)$. Let $I := (x_1^{c_1},\ldots,x_n^{c_n})$, and put $\mu := f^q x_1^{\epsilon_1^+}\cdots x_n^{\epsilon_n^+}$. We will show that
\begin{enumerate}
    \item $\mu \notin I$,
    \item $L^q\mu \in I$,
    \item $\deg(L^{q-1}\mu) \leq \lfloor \frac{\sigma -1}{2} \rfloor$ where $\sigma$ denotes the socle degree of $S/I$.
\end{enumerate}
Items (1) and (2) imply that there exists $0\leq s \leq q-1$ such that $L^s\mu \notin I$ and $L^{s+1}\mu \in I$. Hence $L^s\mu$ is in the kernel of the multiplication map of $L$ on $S/I$. Item (3) implies that this kernel element has degree less than half the socle degree. For an Artinian quotient by a monomial ideal, verifying the WLP is equivalent to checking that multiplication by $L$ has maximal rank \cite{harima2003weak}. Thus, $S/I$ fails the WLP.

\textbf{Claim 1}. We show that $\mu \notin I$.

\textit{Proof of Claim 1.} Let $x^{\u}$ be a monomial appearing in the support of $f$ which does not belong to $J$. Then, $0 \leq u_i \leq b_i$ for $1 \leq i \leq n$. The $x_i$-degree of $x^{q\u} x^{\eps^{+}}$ satisfies
\begin{align*}
    qu_i + \epsilon_i^+ 
    &\leq qb_i+\epsilon_i^+ \\
    &= c_i - q - \epsilon_i^- \\
    &\leq c_i-1.
\end{align*}
Thus, $\mu \notin I$.

\textbf{Claim 2}. We show that $L^q\mu \in I$.

\textit{Proof of Claim 2.} We observe that
\begin{align*}
    L^q\mu = (Lf)^q x^{\eps^+}.
\end{align*}
Since $Lf \in J$, we have that 
\begin{align*}
    (Lf)^q x^{\eps^+} \in J^{[q]}\cdot (x^{\eps^+}) \subset I.
\end{align*}

\textbf{Claim 3}. We show that $\deg(L^{q-1}\mu) \leq \lfloor \frac{\sigma -1}{2} \rfloor$ where $\sigma$ denotes the socle degree of $S/I$.

\textit{Proof of Claim 3.} We compute
\begin{equation}
\label{eq:deg}
    \deg(L^{q-1}\mu) = qd + \sum_{i=1}^{n} \epsilon_i^+ + q -1.
\end{equation}
We compute 
\begin{equation}
\label{eq:socle_deg}
    \sigma = (q-1)n + 2qd + \sum_{i=1}^n \epsilon_i^+ + \sum_{i=1}^n \epsilon_i^-.
\end{equation}
Using equations \eqref{eq:deg} and \eqref{eq:socle_deg}, we compute
\begin{equation}
\label{eqn:eq}
\begin{aligned}
    \sigma -1 -2\deg(L^{q-1}\mu)
    &= (n-2)(q-1)-1 - \left( \sum_{i=1}^n \epsilon_i^+ - \sum_{i=1}^n \epsilon_i^- \right) \\
    &= r_q - \norm{\eps}_1 \\
    &\geq 0.
\end{aligned}
\end{equation}
The identity in equation \eqref{eqn:eq} explains the taxicab radius in Theorem \ref{thm:intro}.
\end{proof}

\subsection{Reverse Inclusion}

Consider the category of finite-length graded $\k[T]$-modules where $T$ acts as a positively graded linear map. In this category, the action of $T$ on the tensor product $M \tensor_\k N$ is required to satisfy:
\begin{align*}
    T(m\tensor n) := (Tm) \tensor n + m\tensor (Tn).
\end{align*}
The graded Han--Monsky representation ring \cite{HanMonsky1993} is the split Grothendieck ring of this category. As an abelian group the graded Han--Monsky ring is generated by isomorphism classes of objects in this category, where $+$ satisfies $[M\oplus N] = [M]+[N]$. Multiplication is given by tensor product. For $r \geq 1$ and $j \in \Z$, we define the element $\delta_r(-j)$ to be the isomorphism class of $\k[T]/(T^r)$ whose generator is shifted to degree $j$. Every indecomposable object in the graded Han--Monsky representation ring is isomorphic to $\delta_r(-j)$ for some $r\geq 1$ and $j \in\Z$. The set $\{\delta_r(-j) : r\geq 1, j \in \Z\}$ forms a $\Z$-basis for the graded Han--Monsky representation ring. The isomorphism class of $\Ac$ as a $\k[T]$-module where $T = x_1+\cdots+x_n$ is represented by $\delta_{c_1}\cdots \delta_{c_n}$ where $\delta_r := \delta_r(0)$.

With this setup, $A_\c$ fails the weak Lefschetz property if and only if $\delta_r(-j)$ is a summand of $\delta_{c_1}\cdots\delta_{c_n}$ satisfying
\begin{equation}
    \label{eq:failure_of_wlp}
    2(r+j-1) \leq \sigma-1
\end{equation}
where $\sigma = \sum_{i=1}^n (c_i -1)$ is the socle degree of $A_\c$.

We recall the following result of \cite{kyomuhangi2024cohomology}.
\begin{lemma}[{\cite[Proposition 5.4]{kyomuhangi2024cohomology}; \cite[Proposition 3.3]{KMRR2026Cohomology}}]
\label{lem:prime_to_p_centering}
Let $c_1,\ldots,c_n$ be positive integers. If $p\nmid r$ and $\delta_r(-j)$ is a summand of $\delta_{c_1}\cdots\delta_{c_n}$, then
\begin{align*}
    r+2j = \sigma+1
\end{align*}
where $\sigma = \sum_{i=1}^{n} (c_i - 1)$ is the socle degree of $\Ac$.
\end{lemma}

The next lemma shows that $\delta_r(-j)$ satisfying \eqref{eq:failure_of_wlp} has length divisible by $p$.

\begin{lemma}
\label{lem:divisible_by_p}
If $\c \in \NL{n}{p}$ and $\delta_r(-j)$ is a summand of $\delta_{c_1}\cdots\delta_{c_n}$ so that equation \eqref{eq:failure_of_wlp} holds, then $p \mid r$.
\end{lemma}

\begin{proof}
Suppose by contradiction that $p\nmid r$. Then, we compute  
\begin{align*}
    \sigma+1 &= r+2j & & \text{(Lemma \ref{lem:prime_to_p_centering})} \\
    &\leq r+2j+(r-1) \\
    &\leq \sigma, & & (\text{equation } \eqref{eq:failure_of_wlp})
\end{align*}
which is a contradiction.
\end{proof}

Let $\c \in \NL{n}{p}$ and $\delta_r(-j)$ be a summand of $\delta_{c_1}\cdots\delta_{c_n}$ which satisfies equation \eqref{eq:failure_of_wlp}. Hence, there exists a homogeneous element $v \in \Ac$ of degree $j$ satisfying $T^{r-1}v \neq 0$ and $T^rv = 0$ so that $\Ac = \k[T]v \oplus B$ for some graded $\k[T]$-submodule $B$. Lemma \ref{lem:divisible_by_p} implies that $r = q\lambda$ where $q = p^e$ and $p\nmid \lambda$. Our goal is to descend $\delta_r(-j)$ to a summand $\delta_\lambda(-j_0)$ of some Artinian complete intersection, in which case we may apply Lemma \ref{lem:prime_to_p_centering}.

To this end, let $Z = T^q = x_1^q + \cdots + x_n^q$. For $\mathbb{\brho} := (\rho_1,\ldots,\rho_n)$ where $0 \leq \rho_i < \min(c_i,q)$, we define 
\begin{align*}
    M_{\brho} := \Span_\k \left\{ x_1^{\rho_1+qk_1}\cdots x_n^{\rho_n+qk_n} : k_i \in \Z, 0 \leq \rho_i +q k_i < c_i \right\}.
\end{align*}
We have that $A_\c \iso \bigoplus_{\brho} M_{\brho}$ as $\k[Z]$-modules where $\deg(Z) = q$. We also have that
\begin{align*}
    \k[T]v = \bigoplus_{0\leq a\leq q-1} \Span_{\k}\{T^{a+qs}v : 0 \leq s \leq \lambda-1\}
\end{align*}
as $\k[Z]$-modules. Let $W_0 := \Span_\k\{ v,Zv,\ldots,Z^{\lambda-1}v \}$. Because $W_0$ is an indecomposable graded $\k[Z]$-module appearing as a direct summand of $A_\c$, the graded Krull--Schmidt theorem implies that some $M_{\brho}$ contains a direct summand isomorphic to $W_0$.

\begin{lemma}
\label{lem:iso_M_rho_as_ci}
For $0 \leq \rho_i < \min(c_i,q)$, let $u_i := \lfloor \frac{c_i-1-\rho_i}{q}\rfloor$. Define 
\begin{align*}
    C_{\brho} := \k[y_1,\ldots,y_n]/(y_1^{u_1+1},\ldots,y_n^{u_n+1}).
\end{align*}
Let $Z$ act on $C_{\brho}$ by $Y := y_1+\cdots+y_n$. Then, $C_{\brho} \iso M_{\brho}$ as $\k[Z]$-modules.
\end{lemma}

\begin{proof}
The map $C_{\brho} \ra M_{\brho}$ sending $y_1^{a_1}\cdots y_n^{a_n}$ to $x_1^{\rho_1+a_1q}\cdots x_n^{\rho_n+a_nq}$ is an isomorphism of $\k[Z]$-modules.
\end{proof}

Since an isomorphic copy of $W_0$ is a direct summand of $M_{\brho}$ for some choice of vector $\brho$, Lemma \ref{lem:iso_M_rho_as_ci} implies that $\delta_\lambda(-j_0)$ is a summand of $\delta_{u_1+1}\cdots \delta_{u_n+1}$ where $j_0$ satisfies 
\begin{equation}
    \label{eqn:formula_for_j_zero}
    j = \sum_{i=1}^n \rho_i + q j_0.
\end{equation}
Moreover, Lemma \ref{lem:prime_to_p_centering} implies 
\begin{equation}
    \label{eqn:lam_plus_j_zero}
    \lambda + 2j_0 = \sum_{i=1}^n u_i +1.
\end{equation}
Our first attempt to approximate $\c$ by a translated center of $\mathscr{B}_q$ is via $q(\o+\u)$. It need not be the case that $\u\in \Gamma_n$, which we will rectify below, but first we record the distance between $\c$ and $q(\o+\u)$.

\begin{lemma}
\label{lem:E_ineq}
For $1 \leq i \leq n$, let $0 \leq \kappa_i \leq q-1$ be such that $c_i-1-\rho_i = q u_i + \kappa_i$. Put $\eta_i := q-1-\kappa_i$ and $E := \sum_{i=1}^n\left( \rho_i + \eta_i \right)$. Then,
\begin{enumerate}
    \item $0 \leq \eta_i \leq q-1$,
    \item $c_i = q(u_i+1) + \rho_i - \eta_i$, and
    \item $\norm{\c-q(\o+\u)}_1 = \sum_{i=1}^n \card{\rho_i - \eta_i} \leq E$.
\end{enumerate}
\end{lemma}

The next lemma gives us a useful equivalent expression for $E := \sum_{i=1}^n\left( \rho_i + \eta_i \right)$.
\begin{lemma}
\label{lem:descr_E}
Let $\Delta := \sigma - 2(j+q\lambda-1)$. Then, $\Delta \geq 1$ and
\begin{equation}
    E = (n-2)(q-1) -\Delta -q(\lambda-1).
\end{equation}
\end{lemma}

\begin{proof}
By equation \eqref{eq:failure_of_wlp}, $\Delta \geq 1$. Lemma \ref{lem:E_ineq} implies that
\begin{equation}
\label{eqn:alt_format_for_sigma}
\begin{aligned}
    \sigma &= \sum_{i=1}^n \left( c_i -1 \right) \\
    &= q \sum_{i=1}^n  u_i  + n(q-1) + \sum_{i=1}^n \rho_i - \sum_{i=1}^n \eta_i.
\end{aligned}
\end{equation}
Substituting equations \eqref{eqn:alt_format_for_sigma} and \eqref{eqn:formula_for_j_zero} into $\Delta =  \sigma - 2(j+q\lambda-1)$, we obtain
\begin{align*}
     \Delta &= \sigma - 2(j+q\lambda-1) \\
     &= q \sum_{i=1}^n u_i  + n(q-1) + \sum_{i=1}^n \rho_i  - \sum_{i=1}^n \eta_i   - 2\left(\sum_{i=1}^n \rho_i  + q j_0 +q\lambda-1 \right) \\
     &= n(q-1) - \sum_{i=1}^n (\rho_i + \eta_i) + q\left( \sum_{i=1}^nu_i -2j_0-2\lambda \right)+2 \\
     &= n(q-1) -E - q(\lambda+1) +2 \quad\quad (\ref{eqn:lam_plus_j_zero}) \\
     &= (n-2)(q-1) -E -q(\lambda-1).
\end{align*}
\end{proof}
The following lemmas will enable us to find $\b \in \Gamma_n$ which is close to $\u$.

\begin{lemma}
\label{lem:even}
With notation as in Section 2.2, we have that $\sum_{i=1}^n u_i + \lambda-1$ is even. For $1 \leq k \leq n$, 
\begin{equation}
    2u_k \leq \sum_{i=1}^n u_i + \lambda-1.
\end{equation}
\end{lemma}

\begin{proof}
Let $w \in C_{\brho}$ be a homogeneous generator corresponding to the summand $\delta_\lambda(-j_0)$. From the Jordan decomposition of $C_{\brho}$ as a $\k[Y]$-module, we have that the image of $w$ in $C_{\brho}/YC_{\brho}$ is non-zero. $C_{\brho}/YC_{\brho}$ is a quotient of the algebra
\begin{align*}
    B_k := \frac{\k[y_1,\ldots,\widehat{y_k},\ldots,y_n]}{(y_1^{u_1+1},\ldots,\widehat{y_k^{u_k+1}},\ldots,y_n^{u_n+1})}.
\end{align*}
Hence the degree of $w$ is at most the socle degree of $B_k$, i.e.,
\begin{equation}
    \label{eq:bnd}
    j_0 \leq \sum_{i=1}^nu_i - u_k.
\end{equation}
Combining equations \eqref{eq:bnd} and \eqref{eqn:lam_plus_j_zero} implies
\begin{align*}
    2u_k \leq 2 \left( \sum_{i=1}^nu_i - j_0 \right) = \sum_{i=1}^nu_i + \lambda -1,
\end{align*}
and the right-hand side is  even.
\end{proof}

\begin{lemma}
\label{lem:choose_b}
Let $\u \in \Z_{\geq 0}^n$ and $\lambda \in \Z_{\geq1}$. Suppose $\sum_{i=1}^n  u_i + \lambda -1$ is even and $2u_j \leq \sum_{i=1}^n  u_i+\lambda-1$ for all $1 \leq j \leq n$. Then, there exists $\b \in \Gamma_n$ so that $0 \leq b_j \leq u_j$ for all $j$ and $\sum_{i=1}^n(u_i-b_i) \leq \lambda -1$.
\end{lemma}

\begin{proof}
Let $M = \max_i u_i$. Let $1 \leq k \leq n$ be an integer so that $M = u_k$.

\ul{Case 1.} Suppose that $2 M \leq \sum_{i=1}^n u_i$ and $\sum_{i=1}^n u_i$ is even. Then, $\b := \u \in \Gamma_n$.

\ul{Case 2.} Suppose that $2 M \leq \sum_{i=1}^n u_i$ and $\sum_{i=1}^n u_i$ is odd. Put $b_i := u_i$ for $i \neq k$, and put $b_k := u_k-1$. Since $\sum_{i=1}^n u_i$ is odd, $M \neq 0$, and hence $\b \in \Z_{\geq0}^n$. Next, we have that $2\max_i b_i \leq 2 u_k$. Since $2u_k \leq \sum_{i=1}^n u_i$ (assumption) and $2u_k$ is even and $\sum_{i=1}^n u_i$ is odd (assumption), the difference in parity implies that
\begin{align*}
    2u_k \leq \sum_{i=1}^n u_i-1 = \norm{\b}_1.
\end{align*}
Thus, $2\max_i b_i \leq \norm{\b}_1$. Since in addition $\norm{\b}_1$ is even, $\b \in \Gamma_n$ by \eqref{eqn:gamma_n}.

Since $\sum_{i=1}^n u_i + \lambda -1$ is even (assumption) and $\sum_{i=1}^n u_i$ is odd (assumption), it follows that $\lambda-1$ is odd. Hence $\lambda-1 \geq 1$, and it follows that 
\begin{align*}
    \norm{\u-\b}_1 =1 \leq \lambda-1.
\end{align*}

\ul{Case 3.} Suppose $\sum_{i=1}^nu_i < 2M$. Let $b_i := u_i$ for $i \neq k$ and $b_k := \sum_{i=1}^n u_i-M$. By construction, $\b \in \Z_{\geq 0}^n$ and $\sum_{i=1}^n b_i = 2(\sum_{i=1}^n u_i-M)$. For $j \neq k$, we have
\begin{align*}
    b_j = u_j \leq \sum_{i\neq k} u_i = \sum_{i=1}^n u_i-M = b_k.
\end{align*}
Thus, for $1 \leq j \leq n$, 
\begin{align*}
    2 b_j \leq 2 b_k = 2\left(\sum_{i=1}^n u_i-M\right),
\end{align*}
and it follows that $\b \in \Gamma_n$.

The hypothesis $\sum_{i=1}^n  u_i < 2M$ implies that $b_k \leq u_k$. The hypothesis $2u_k \leq \sum_{i=1}^n u_i + \lambda-1$ implies $2M - \sum_{i=1}^n u_i \leq \lambda-1$. We compute that
\begin{align*}
    \sum_{i=1}^n (u_i - b_i) = 2M - \sum_{i=1}^n  u_i \leq \lambda-1,
\end{align*}
which completes the proof.
\end{proof}

We are now ready to show that $\c \in \mscr{B}_q$.

\begin{proof}[Proof of the reverse inclusion in Theorem \ref{thm:intro}]
Continue with the notation of Section 2.2. By Lemmas \ref{lem:even} and \ref{lem:choose_b}, we can find $\b \in \Gamma_n$ so that $0 \leq b_i \leq u_i$ and $\norm{\u-\b}_1 \leq \lambda-1$. Put $\eps := \c-q(\o+\b).$ We compute
\begin{align*}
    \norm{\eps}_1 &= \norm{\c - q(\o+\b)}_1 \\
    &\leq \norm{\c-q(\o+\u)}_1 + q\norm{\u-\b}_1 & & \\
    &\leq E + q(\lambda-1) & & (\text{Lemmas } \ref{lem:E_ineq}, \; \ref{lem:choose_b}) \\
    &= (n-2)(q-1) - q(\lambda-1) - \Delta + q(\lambda-1) & & (\text{Lemma } \ref{lem:descr_E}) \\
    &\leq (n-2)(q-1) -1 & & (\text{Lemma } \ref{lem:descr_E}) \\
    &= r_q.
\end{align*}
We compute
\begin{align*}
    \epsilon_i 
    &= c_i-q-qb_i & & \\
    &= (qu_i+q+\rho_i-\eta_i) -q-qb_i & & (\text{Lemma \ref{lem:E_ineq}}) \\
    &\geq q(u_i - b_i) + \rho_i +1 - q & & (\text{Lemma } \ref{lem:E_ineq}(1)) \\
    &\geq 1-q.
\end{align*}
By Lemma \ref{lem:equiv_char}, it follows that $\c \in \mscr{B}_q$, which proves the reverse containment.
\end{proof}

\section*{Acknowledgments}
The author was partially supported by NSF grant DMS-2342256. The author found computations with Macaulay2 \cite{M2} helpful. During conversations with ChatGPT, the author became aware of the reference \cite{Dhand2014}, which led to a simplification of the author's earlier arguments for the forward inclusion in Theorem \ref{thm:intro}. The author also asked ChatGPT to use \cite[Proposition 5.4]{kyomuhangi2024cohomology} to explore a proof of the reverse inclusion in Theorem \ref{thm:intro}. ChatGPT provided a proof outline, which the author subsequently developed into a rigorous proof. ChatGPT was then used to refine and check details of the resulting argument. The author takes full responsibility for the correctness of all mathematical content.

\bibliographystyle{alpha}
\bibliography{bibliography}

\end{document}